\documentclass[12pt]{amsart}
\usepackage[all]{xy}
\usepackage{amsmath}
\usepackage{amsthm}
\usepackage{amsfonts}
\usepackage{amssymb}
\usepackage{amscd}
\usepackage{epic}
\usepackage{eepic}
\usepackage{verbatim}
\newtheorem{theorem}{Theorem}[section]
\newtheorem*{theorem*}{Theorem}
\newtheorem{proposition}[theorem]{Proposition}
\newtheorem{lemma}[theorem]{Lemma}
\newtheorem{corollary}[theorem]{Corollary}

 \theoremstyle{definition}
\newtheorem{remark}[theorem]{Remark}

\renewcommand{\(}{\bigl(} \renewcommand{\)}{\bigr)}

\newcommand{\tens}{\otimes}

\newcommand{\leftexp}[2]{{\vphantom{#2}}^{#1}{#2}}

\newcommand{\id}{\mathrm{id}}

\newcommand{\op}{^{\mathrm{op}}}
\newcommand{\ra}{\rightarrow}
\newcommand{\xra}{\xrightarrow}
\newcommand{\CH}{\operatorname{CH}}
\renewcommand{\Im}{\operatorname{Im}}

\newcommand{\Ker}{\operatorname{Ker}}
\newcommand{\Pic}{\operatorname{Pic}}

\newcommand{\ind}{\operatorname{ind}}

\newcommand{\Int}{\operatorname{Int}}
\newcommand{\res}{\operatorname{res}}
\newcommand{\cor}{\operatorname{cor}}
\newcommand{\Br}{\operatorname{Br}}

\newcommand{\Spec}{\operatorname{Spec}}

\newcommand{\Gal}{\operatorname{Gal}}

\newcommand{\gGL}{\operatorname{\mathbf{GL}}}
\newcommand{\gPGL}{\operatorname{\mathbf{PGL}}}

\newcommand{\End}{\operatorname{End}}
\newcommand{\Hom}{\operatorname{Hom}}
\newcommand{\Mor}{\operatorname{Mor}}

\newcommand{\Aut}{\operatorname{Aut}}

\newcommand{\spann}{\operatorname{span}}

\newcommand{\Sym}{\operatorname{Sym}}
\newcommand{\red}{\operatorname{red}}

\newcommand{\Tr}{\operatorname{Tr}}
\newcommand{\Trd}{\operatorname{Trd}}

\newcommand{\GCD}{\operatorname{g.c.d.}}
\newcommand{\NK}{\operatorname{N_{K/F}(K^\times)}}
\newcommand{\NKL}{\operatorname{N_{KL/L}\((KL)^\times\)}}
\newcommand{\rank}{\operatorname{rank}}
\renewcommand{\P}{\mathbb{P}}
\newcommand{\Z}{\mathbb{Z}}

\newcommand{\gm}{\mathbb{G}_m}

\newcommand{\cC}{\mathcal C}

\newcommand{\cI}{\mathcal I}
\newcommand{\cJ}{\mathcal J}
\newcommand{\cO}{\mathcal O}

\newcommand{\cL}{\mathcal L}
\newcommand{\cP}{\mathcal P}

\title{Del Pezzo Surfaces of degree 6 over an Arbitrary Field}
\author{Mark Blunk}
\date{}
\address{Department of Mathematics, University of California,
        Los Angeles, CA 90095-1555} \email {mblunk@math.ucla.edu}

\begin{document}
\begin{abstract}
We give a characterization of all del Pezzo surfaces of degree 6
over an arbitrary field $F$. A surface is determined by a pair of
separable algebras. These algebras are used to compute the Quillen
$K$-theory of the surface. As a consequence, we obtain an index
reduction formula for the function field of the surface.
\end{abstract}
\maketitle
\section{Introduction}
If $X$ is an algebraic variety defined over an arbitrary field
$F$, a common method (cf. the introduction of \cite{Vos98}) for
learning various properties of $X$ is to first study $\overline{X}
:= X \times_{\Spec F} \Spec (\overline{F})$, the extension of
scalars of $X$ to a separable closure $\overline{F}$ of $F$, and
then to study the action of the Galois group
$\Gal(\overline{F}/F)$ on algebraic groups and other algebraic
objects associated to $\overline{X}$. This is particularly useful
when dealing with a class of varieties that all become isomorphic
over $\overline{F}$, e.g. Severi-Brauer varieties or involution
varieties. A Severi-Brauer variety is determined by a central
simple $F$-algebra $A$, and an involution variety is determined by
a central simple $F$-algebra $A$ with an orthogonal involution of
the first kind $(A, \sigma)$. In either case this algebraic data
determines geometrical and topological information about the
corresponding variety.  In particular the Quillen $K$-groups of
the variety are determined the algebra in the Severi-Brauer
example, and the algebra with involution in the involution variety
example. This was proved for Severi-Brauer varieties and
involution varieties in \cite{Qui72} and \cite{Tao94},
respectively. Panin proved in \cite{Pan94} a more general theorem
computing the $K$-theory of projective homogeneous varieties,
which contains both examples as special cases. In all of these
examples, as in this paper, the action of algebraic groups plays a
significant role. An immediate consequence of this computation of
the $K$-theory is an index reduction formula, which determines how
extending scalars of a division $F$-algebra to the function field
of the variety reduces the index of the algebra. In this paper we
will study del Pezzo surfaces of degree 6 over $F$, obtaining
similar results.

A del Pezzo surface $S$ is a smooth projective surface over a
field $F$ such that the anti-canonical bundle $\omega_S^{-1}$ is
ample. The degree (the self-intersection number of $\omega_S$) of
any such surface can be any integer between 1 and 9. Such
varieties were discussed in \cite{ColKarMer07}, \cite{Cor05}, and
\cite{Vos82}. As mentioned in some of these references, a del
Pezzo surface of degree 6 is a toric variety for a particular two
dimensional torus, which we will describe below. We explore this
toric structure in Section \ref{secToricVar}. The result is
Theorem \ref{torus}, a classification of all such surfaces up to
isomorphism preserving the action of the torus. Section
\ref{secMain} contains the main result of the paper, Theorem
\ref{main}, where it is proved that a del Pezzo surface of degree
6 is determined by a pair $B$ and $Q$ of separable $F$-algebras,
with centers $K$ and $L$ \'{e}tale quadratic and cubic over $F$
respectively, and both containing $K \tens_F L$ as a subalgebra.
Moreover, $\cor_{K/F}(B)$ and $\cor_{L/F}(Q)$ must be split. As an
immediate corollary of Theorems \ref{torus} and \ref{main}, we
give a necessary and sufficient condition in terms of $B$ and $Q$
for determining when the corresponding surface will have a
rational point. 

In Section \ref{secK_0}, we relate the algebras $B$ and $Q$ to the
endomorphism rings of locally free sheaves on the associated del
Pezzo surface $S$. These sheaves are used in Theorem \ref{zero} to
relate the Quillen $K$-theory of $S$ to that of $B$ and $Q$, by
showing that the algebra $A = F \times B \times Q$ is isomorphic
to $S$ in a certain $K$-motivic category $\cC$, constructed in
\cite{Pan94}. This implies that for all $n$,
\begin{displaymath}
K_n(S) \cong K_n(A) = K_n(F) \oplus K_n(B) \oplus K_n(Q).
\end{displaymath}
As a corollary we obtain an index reduction formula for the
function field of $S$.

I would like to thank my advisor Alexander Merkurjev, who posed
this question to me, and answered several of my questions which
developed along the way.

We use the following notations and conventions:

An  $F$-variety is a separated scheme of finite type over $\Spec
(F)$.

$\overline{F}$ will denote a separable closure of $F$.

An $F$-algebra A is separable if $A \tens_F L$ is semisimple for every
field extension $L$ of $F$. Such an algebra is Azumaya over its
center, which is an \'{e}tale extension of $F$.

$\Gamma$ will denote the group $\Gal (\overline{F}/F)$.

For any $F$-variety $X$ and any field extension $E$ of $F$, we
will denote $X \times_{\Spec F} \Spec(E)$ (resp. $X \times_{\Spec
F} \Spec(\overline{F})$) by $X_E$ (resp. $\overline{X}$).

For any separable $F$-algebra $A$ and any \'{e}tale extension $E$
of $F$, we will denote $A \tens_F E$ (resp. $A \tens_F
\overline{F}$) by $A_E$ (resp. $\overline{A}$).

If $D$ is a Cartier divisor on a variety $X$, $\cL(D)$ will denote
the corresponding invertible sheaf on $X$.

For any variety $X$ and any separable algebra $A$, $\mathbf{P}(X;
A)$ will denote the exact category of left $A \tens_F
\cO_X$-modules which are locally free $\cO_X$-modules. We will
denote $ \mathbf{P}(X; F)$ (resp. $\mathbf{P}(\Spec F; A)$) by
$\mathbf{P}(X)$ (resp. $\mathbf{P}(A)$).

For any integer $n$, $K_n(X; A)$ will denote the Quillen group
$K_n(\mathbf{P}(X; A))$. As above, we will denote $K_n(X; F)$ by
$K_n(X)$ and $K_n(\Spec F;A)$ by $K_n(A)$.

For any algebraic torus $T$, $\widehat{T}$ will denote the
$\Gamma$-module of characters $\Hom_{\overline{F}}(\overline{T},
\mathbb{G}_{m, \overline{F}})$.

\section{Toric Varieties}
\label{secToricVar}

We first recall from \cite{Har77}, \cite{Man74}, and \cite{Vos82}
some basic properties of the variety $\widetilde{S}$, the blow up
of $\P^2$ at the 3 non-collinear points $[1 : 0 : 0]$, $[0 : 1 :
0]$, and $[0 : 0 : 1]$. The variety $\widetilde{S}$ can be
realized as a closed subvariety of $\P^2 \times \P^2$, defined by
the equations $x_0 y_0 = x_1 y_1 = x_2 y_2$. The projection onto
the first factor of $\P^2$ is the blow down of the three lines
$m_0 = \{x_1 = x_2 = 0\}$, $m_1 = \{x_0 = x_2 = 0\}$, and $m_2 =
\{x_0 = x_1 = 0\}$. Similarly, the projection onto the second
factor of $\P^2$ is the blow down of the three lines $l_0 = \{y_1
= y_2 = 0\}$, $l_1 = \{y_0 = y_2 = 0\}$, and $l_2 = \{y_0 = y_1 =
0\}$.

\begin{proposition} Let $\widetilde{S}$ be the blow up of $\P^2$
at the three points $[1 : 0 : 0]$, $[0 : 1 : 0]$, and $[0 : 0 :
1]$.
\begin{enumerate}
\item[i.] The variety $\widetilde{S}$ is a del Pezzo surface of
degree 6 over $F$, and if $F$ is separably closed, any del Pezzo
surface $S$ of degree 6 over $F$ is isomorphic to $\widetilde{S}$.

\item[ii.] The group $\CH^1(\widetilde{S})$ is generated by the
lines $l_0$, $l_1$, $l_2$, $m_0$, $m_1$, and $m_2$.

\item[iii.] The intersection pairing on $\CH^1 (\widetilde{S})$ is
determined by the following relations: $l_i^2 = - 1$, $m_i^2 =
-1$, $l_im_j = 1$, and $l_im_i=l_il_j = m_im_j = 0$, for distinct
$i, j \in \{0, 1, 2\}$.

\item[iv.] The group $\CH^2(\widetilde{S})$ is cyclic, generated
by the class of any rational point.

\end{enumerate}
\label{delPezzo}
\end{proposition}

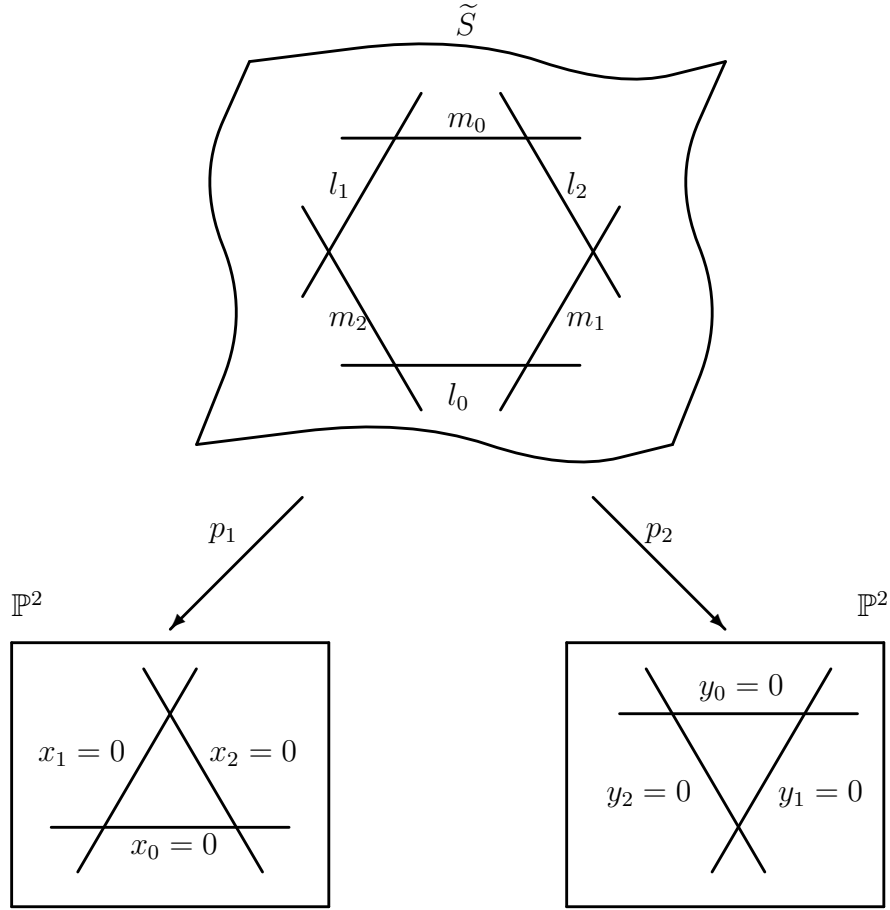
\begin{figure}
\begin{picture}(300,400)(10,20)
\Thicklines

\put (35 ,100 ){\drawline(-20,0)(70,0) \drawline(-10, -17)(35, 60)
  \drawline(15, 60)(60,-17)}

\put(40,90){ \drawline(-40, -20)(80, -20)(80, 80)(-40, 80)(-40, -20)}

\put (250,100 ){\drawline(-20,43)(70,43) \drawline(-10, 60)(35, -17)
  \drawline(15, -17)(60,60) }

\put(250,90){ \drawline(-40, -20)(80, -20)(80,80)(-40,80)(-40,-20)}


\put (120, 275){\drawline(-10, 26)(35, 103) \drawline(5, 86)(95, 86)
  \drawline(65, 103)(110, 26) \drawline(110, 60)(65, -17)
  \drawline(95, 0)(5, 0) \drawline(35, -17)(-10, 60)}

\put(70,245){\spline(20, 145)(100, 155)(160, 135)(200, 145)
\spline(0, 0)(80, 10)(140, -10)(180, 0)
\spline(0, 0)(20, 50)(0, 100) (20, 145)
\spline(180, 0)(200, 50)(180, 100) (200, 145)}

\put(110,225){\vector(-1, -1){50}}

\put(220,225){\vector(1, -1){50}}


\put(0,180){$\P^2$} \put(10, 125){$x_1 = 0$} \put(75, 125){$x_2
=0$} \put(45,90){$x_0 = 0$}

\put(320, 180){$\P^2$}
\put(260, 150){$y_0 = 0$}
\put(290, 110){$y_1 = 0$}
\put(225, 110){$y_2 = 0$}

\put(168,400){$\widetilde{S}$}

\put(165, 260){$l_0$}
\put(120, 340){$l_1$}
\put(210, 340){$l_2$}
\put(165, 365){$m_0$}
\put(210, 290){$m_1$}
\put(120, 290){$m_2$}

\put(75, 210){$p_1$} \put(240, 210){$p_2$}

\end{picture}
\caption[]{The Hexagon of Lines.}
\end{figure}

As mentioned in \cite{ColKarMer07}, there is an action of the
torus $\widetilde{T} = \gm^3/\gm$ on $\P^2$,  described by:
\begin{equation*}(t_0, t_1, t_2) \cdot [x_0, x_1, x_2; y_0, y_1, y_2] = [t_0
x_0, t_1 x_1, t_2 x_2; t_0^{-1} y_0, t_1^{-1} y_1, t_2^{-1}y_2].
\end{equation*}
Here $\gm$ embeds into $\gm^3$ diagonally. This action sends
$\widetilde{S}$ to itself, and is faithful and transitive on the
open subset $\widetilde{U}$ of $\widetilde{S}$, the complement of
the subvariety defined by the equation $x_0 x_1 x_2 y_0 y_1 y_2 =
0$. This closed subvariety has 6 irreducible components, the lines
$l_0$, $l_1$, $l_2$, $m_0$, $m_1$, and $m_2$, which by the
proposition are arranged in a hexagon. Thus $\widetilde{S}$ is a
$\widetilde{T}$-toric variety, with fan dual to the hexagon of
lines. There is also an action of the symmetric groups $S_2$ and
$S_3$ on $\widetilde{S}$. The nontrivial element of $S_2$ acts on
$\P^2 \times \P^2$ by interchanging the $x_i$ and $y_i$, and the
$S_3$ action on $\P^2 \times \P^2$ arises from the diagonal action
of $S_3$ on the coordinates $x_0$, $x_1$, $x_2$ and $y_0$, $y_1$,
$y_2$. The $S_2$ and $S_3$ actions commute with each other, and
both groups send $\widetilde{S} \subset \P^2 \times \P^2$ to
itself. Therefore they induce an action of $S_2 \times S_3$ on
$\widetilde{S}$, preserving the set of lines $l_0$, $l_1$, $l_2$,
$m_0$, $m_1$, and $m_2$, and thus inducing an isomorphism from
$S_2 \times S_3$ onto the automorphism group of the hexagon of
lines. The torus $\widetilde{T}$ is the connected component of the
identity of the algebraic group $\Aut_F(\widetilde{S})$ of
automorphisms of $\widetilde{S}$, and $S_2 \times S_3$ is the
group of connected components. The action of $S_2 \times S_3$ on
$\widetilde{S}$ define a section $S_2 \times S_3 \ra \Aut_F
(\widetilde{S})$, so we have the following split exact sequence of
algebraic groups:
\[1 \ra \widetilde{T} \ra \Aut_F(\widetilde{S}) \ra S_2 \times S_3
\ra 1.\]

We have another way to realize $\widetilde{S}$ as a closed
subvariety of a product of projective spaces. Define
$f_i:\widetilde{S}: \ra \P^1$ for $i = 0, 1, 2$ by
\begin{align*}
f_0([x_0: x_1: x_2 ; y_0: y_1: y_2]) =  [x_1 : x_2]\ \
\mathrm{or}\ [y_2 : y_1] \\
f_1([x_0: x_1: x_2 ; y_0: y_1: y_2]) = [x_2 : x_0]\ \ \mathrm{or}\
[y_0 : y_2] \\
f_2([x_0: x_1: x_2 ; y_0: y_1: y_2]) = [x_0 : x_1]\ \ \mathrm{or}\
[y_1 : y_0].
\end{align*}
Each $f_i$ is well defined, as the two definitions agree on the
overlap, and thus is a morphism of varieties. These morphisms
define a morphism $f: \widetilde{S} \ra \P^1 \times \P^1 \times
\P^1$. If we denote the bi-homogeneous coordinates of $\P^1 \times
\P^1 \times \P^1$ by $X_0$, $X_1$, $Y_0$, $Y_1$, $Z_0$, and $Z_1$,
it can be shown that $f$ maps $\widetilde{S}$ isomorphically onto
the hypersurface of $\P^1 \times \P^1 \times \P^1$ defined by the
equation $X_0Y_0Z_0 = X_1Y_1Z_1$. The morphism $f$ sends
$\widetilde{T}$ to the torus
\begin{displaymath}\Ker ( \gm^2/\gm \times \gm^2/\gm \times \gm^2/\gm \xra{m} \gm^2/\gm),
\end{displaymath} where $m((t_0,t_1), (t'_0, t'_1), (t''_0, t''_1))
= (t_0t'_0t''_0, t_1t'_1t''_1)$.


Now let $S$ be a del Pezzo surface of degree 6 over an arbitrary
field $F$. Then $\overline{S}$ is a del Pezzo surface of degree 6
over $\overline{F}$, and thus by Proposition \ref{delPezzo} is
isomorphic over $\overline{F}$ to $\overline{\widetilde{S}}$. So
$S$ is an $F$-form of $\widetilde{S}$. As the six lines of the
hexagon form a full set of exceptional curves in $\overline{S}$,
the action of $\Gamma$ on $\overline{S}$ is globally stable on the
set of lines of the hexagon. Therefore, there is an open
subvariety $U$ whose complement $Z$ is isomorphic over
$\overline{F}$ to the hexagon of lines. The action of $\Gamma$ on
$\overline{Z}$ permutes its irreducible components, inducing an
action of $\Gamma$ on the hexagon.

Let $T$ denote the connected component of the identity of
$\Aut_F(S)$. The group of connected components $G$ of $\Aut_F(S)$
is an \'{e}tale group scheme: it is the group scheme determined
(as in Proposition 20.16 of \cite{KnuMerRosTig98}) by the
automorphism group of the hexagon of lines, with continuous
$\Gamma$-action on this finite group as in the previous paragraph.
So $T$ is a torus, $S$ is a $T$-toric variety, with an open set
$U$ which is a  $T$-torsor, and $\Gamma$-action on the fan
determined by the \'{e}tale group scheme $G$.

This $\Gamma$-action on the hexagon determines a homomorphism
$\gamma: \Gamma \ra S_2 \times S_3$. Projecting onto either factor
yields cocycles with values in $S_2$ and $S_3$, and thus $\gamma$
determines a pair $(K, L)$, where $K$ and $L$ are \'{e}tale
quadratic and cubic extensions of $F$, respectively. Note that
while the fan, dual to the hexagon of lines, is the same for all
del Pezzo surfaces of degree 6 over $F$, the possible
$\Gamma$-actions on the fan are in a one-to-one correspondence
with pairs $(K, L)$.

For a fixed cocycle $\gamma$ (i.e. a fixed pair $(K, L)$), we will
classify all del Pezzo surfaces $S$ of degree 6 where the
$\Gamma$-action on $\overline{Z} \subset \overline{S}$ is
determined by $\gamma$.

We have from \cite{Vos82}  the following short exact sequence of
$\Gamma$-modules:
\begin{equation*}
0 \ra \widehat{T} \ra \Z[KL/F] \ra \Pic(\overline{S}) \ra 0.
\label{PicardModule}
\end{equation*}
Here $KL$ denotes the algebra $K \tens_F L$, and $\Z[KL/F]$ is the
lattice of the six lines of $\overline{Z}$. The homomorphism
$\Z[KL/F] \ra \Pic(\overline{S})$ takes a line to the
corresponding Cartier divisor on $\overline{S}$. As described in
\cite{ColKarMer07}, this short exact sequence can be extended into
the exact sequence
\begin{equation}
0 \ra \widehat{T} \ra \Z[KL/F] \ra \Z[K/F] \oplus \Z[L/F] \ra \Z
\ra 0. \label{GammaModule}
\end{equation}
Here $\Z[L/F]$ is the lattice of pairs of opposite lines, and
$\Z[K/F]$ is the lattice of triangles, where each triangle is a
triple of skew lines. The homomorphism $\Z[KL/F] \ra \Z[L/F]$
sends each line to the pair containing it, and the homomorphism
$\Z[KL/F] \ra \Z[K/F]$ sends each line to the triangle containing
it. The homomorphism $\Z[K/F] \oplus \Z[L/F] \ra \Z$ is the
difference of the augmentation maps. This sequence induces the
following short exact sequence of $\Gamma$-modules:
\begin{equation}
0 \ra \widehat{T} \ra \Z[KL/F]/\Z \ra \Z[K/F]/\Z \oplus \Z[L/F]/\Z
\ra 0. \label{GammaModule2}
\end{equation}
where $\Z$ embeds into $\Z[K/F]$, $\Z[L/F]$, and $\Z[KL/F]$
diagonally.

In analogy with $R^{(1)}_{K/F}(\gm) := \Ker (N_{K/F}: R_{K/F}(\gm)
\ra \gm)$, we define the following algebraic $F$-groups:
\begin{align*}
\mathbf{G}_L & := \Ker( N_{KL/L}:R_{KL/F}(\gm) \ra R_{L/F}(\gm))
\\
\mathbf{G}_K & := \Ker (N_{KL/K}: R_{KL/F} (\gm) \ra R_{K/F}(\gm))
\end{align*}
These groups are $F$-tori, dual to the $\Gamma$-modules
$\Z[KL/F]/\Z[L/F]$ and $\Z[KL/F]/ \Z[K/F]$, where $\Z[K/F]$ and
$\Z[L/F]$ are diagonally embedded in $\Z[KL/F]$. The embeddings of
$R_{K/F}(\gm)$ and $R_{L/F}(\gm)$ into $R_{KL/F}(\gm)$ induce
embeddings $R^{(1)}_{K/F}(\gm) \ra \mathbf{G}_L$ and
$R^{(1)}_{L/F}(\gm) \ra \mathbf{G}_K$. The description of
$\widetilde{T} \subset \widetilde{S} \subset \P^2 \times \P^2$
above descends to the following exact sequence:
\begin{equation}
1 \ra R^{(1)}_{K/F}(\gm) \ra \mathbf{G}_L \ra T \ra 1.
\label{eq:toric2}
\end{equation}
Similarly, the description of $f(\widetilde{T}) \subset \P^1
\times \P^1 \times \P^1$ above descends to
\begin{equation*}
1 \ra R^{(1)}_{L/F}(\gm) \ra \mathbf{G}_K \ra T \ra 1.
\end{equation*}
We will use these sequences in Section
\ref{secK_0}.

Finally, from (\ref{GammaModule}) and (\ref{GammaModule2}), we
have corresponding sequences of $F$-tori:
\begin{equation}
1 \ra \gm \ra R_{K/F}(\gm) \times R_{L/F}(\gm) \ra R_{KL/F}(\gm)
\ra T \ra 1, \label{eq:toric}
\end{equation}
and
\begin{displaymath}
1 \ra R^{(1)}_{K/F}(\gm) \times R^{(1)}_{L/F}(\gm) \xra{\phi}
R^{(1)}_{KL/F}(\gm) \ra T \ra 1.
\end{displaymath}
Recall that for $E = K$, $L$, and $KL$,
\begin{align*}H^1(F, R^{(1)}_{E/F}(\gm)) =& F^\times / N_{E/F}(E^\times) \\
H^2(F, R^{(1)}_{E/F}(\gm)) =& \Ker (\cor_{E/F}:\Br(E) \ra
\Br(F)).
\end{align*}
Moreover, as $N_{KL/F}\((KL)^\times\)$ is a subgroup of $\NK$ and
$N_{L/F}(L^\times)$, it follows that the restriction of the
homomorphism of $H^1$ groups induced by $\phi$ to either factor is
just factoring out the corresponding subgroup of the quotient, and
thus $\phi$ will be surjective.  Therefore, by the induced long
exact sequence in cohomology, we obtain the following exact
sequence:
\begin{displaymath}1 \ra H^1(F, T) \ra \Ker(\cor_{K/F}) \times
\Ker (\cor_{L/F}) \ra \Ker (\cor_{KL/F}).\end{displaymath} where
the last homomorphism sends a pair $(x, y)$ to $\res_{KL/K}(x) -
\res_{KL/L}(y) \in \Br(KL)$.

Let $C_1$ be the set of $K$-algebra isomorphism classes of Azumaya
$K$-algebras $B$ of rank 9 such that $B_L = B \tens_K KL$ and
$\cor_{K/F}(B)$ are split, $C_2$ the set of $L$-algebra
isomorphism classes of Azumaya $L$-algebras $Q$ of rank 4 such
that $Q_K = Q \tens_L KL$ and $\cor_{L/F}(Q)$ are split, and set
$C = C_1 \times C_2$. Then $C$ is a pointed set with distinguished
element $(M_3(K), M_2(L))$, and the map $\psi: C \ra
\Ker(\cor_{K/F}) \times \Ker (\cor_{L/F})$ sending a pair $(B, Q)$
to $([B], [Q])$ is a morphism of pointed sets. Moreover,
$\res_{KL/K}([B]) = [B \tens_K KL]$ and $\res_{KL/L}([Q]) = [Q
\tens_L KL]$ are trivial, so it follows that $\psi$ maps into
$H^1(F, T)$.

\begin{theorem}
$\psi:C \ra H^1(F,T)$ is an isomorphism of pointed sets.
\end{theorem}

\begin{proof}
If $\psi(B, Q) = \psi (B', Q')$, then $[B] = [B'] \in
\Ker(\cor_{K/F}) \subset \Br(K)$. Then $B$ and $B'$ are similar
Azumaya $K$-algebras of the same rank, and thus must be isomorphic
as $K$-algebras. Similarly, $Q$ and $Q'$ are isomorphic, so that
$\psi$ is injective.

Now let $(x, y) \in H^1(F, T)$, so that $(x, y) \in
\Ker(\cor_{K/F}) \times \Ker(\cor_{L/F})$, and $\res_{KL/K}(x) =
\res_{KL/L}(y)$. This implies that \begin{align*} 3x &=
\cor_{KL/K}(\res_{KL/K}(x))  \\
& = \cor_{KL/K}(\res_{KL/L}(y)) = \res_{K/F}(\cor_{L/F}(y)) =
0.\end{align*} Similarly, 2y = 0. Thus $\res_{KL/K}(x) =
\res_{KL/L}(y)$ has order divisible by 2 and 3, and therefore is
trivial. If $L$ is not a field, then $L = F \times E$, where $E$
is an \'{e}tale quadratic extension of $F$. Then $\res_{KL/K}(x) =
(x, \res_{E \tens_F K/K}(x))$, and so $\res_{KL/K}(x) = 0$ implies
$x = 0$. If $L$ is a field, then $x$ is split by a field extension
of degree 3 (If $K = F \times F$, and $x = (x_1, x_2) \in \Br(K) =
\Br(F) \times \Br(F)$ is split by $KL$ if and only if $x_1$ and
$x_2$ are split by $L$). Thus for all possible $K$ and $L$, there
is an Azumaya $K$-algebra $B$ of rank 9 that represents $x$ in
$\Br(K)$. Since $\res_{KL/K}(x)$ and $\cor_{K/F}(x)$ are trivial,
$B \tens_K KL$ and $\cor_{K/F}(B)$ are split. Similarly, there is
an Azumaya $L$-algebra $Q$ of rank 4 which represents $y$ such
that $Q \tens_L KL$ and $\cor_{L/F}(Q)$ are split. Then $(B, Q)
\in A$, and $\phi(B,Q) = (x,y)$, so $\psi$ is surjective.
\end{proof}

\begin{remark}
As $KL$ is an \'{e}tale algebra of degree 3 over $K$, if $KL$
splits $B$, then $KL$ can be embedded as a subalgebra of $B$.
Similarly, $KL$ can be embedded as a subalgebra of $Q$. If $(B, Q)
= (B', Q')$ in $C = H^1(F,T)$, then any $K$-isomorphism from $B$
to $B'$ sends $KL \subset B$ to a subalgebra of $B'$ isomorphic to
$KL$. Moreover, if we choose a fixed embedding of $KL$ into both
$B$ and $B'$, by applying Skolem-Noether to $B'$ we can find an
isomorphism from $B$ to $B'$ which restricts to the identity on
$KL$. Similarly, we may assume that $Q$ to $Q'$ are isomorphic via
an isomorphism which is the identity on $KL$.

It follows that if $K$ and $L$ are \'{e}tale quadratic and cubic
extensions of $F$ respectively, and $T$ is the two dimensional
torus induced from $K$ and $L$ as in the exact sequence
(\ref{eq:toric}), then elements of $H^1(F, T)$ are determined by
triples $(B, Q, KL)$, where $B$ is an Azumaya $K$-algebra of rank
9 such that $\cor_{K/F} (B)$ is split, $Q$ is an Azumaya algebra
over $L$ of rank 4 such that $\cor_{L/F} (Q)$ is split, and we
have a fixed embedding of $KL$ as a subalgebra into both $B$ and
$Q$. Two triples $(B, Q, KL)$ and $(B', Q', KL)$ will determine
the same element of $H^1(F, T)$ if there are $KL$-algebra
isomorphisms from $B$ to $B'$ and $Q$ to $Q'$.
\end{remark}

If $S$ is a del Pezzo surface of degree 6, and if $T$ is the
connected component of the identity of $\Aut_F (S)$, $S$ is a
$T$-toric variety, with $\Gamma$-action on the fan induced by the
$\Gamma$-action $\gamma$ on the connected components of
$\overline{Z}$, the hexagon of lines. The $T$-torsors $U \subset
S$ is determined by an element of the pointed set $H^1(F,T)$. Two
surfaces $S$ and $S'$ will be isomorphic as toric varieties if and
only if $T$ and $T'$ are isomorphic as algebraic groups, and there
is an isomorphism from $S$ to $S'$ which preserves the action of
$T \cong T'$ on $S$ and $S'$, thus inducing isomorphisms
$\Gamma$-actions on the fan and isomorphisms of the $T$-torsors
determining $S$ and $S'$. Thus we have proved the following:

\begin{theorem}
Let $S$ be a del Pezzo surface of degree 6, and $T$ be the
connected component of the identity of the group $\Aut_F(S)$. Then
$S$ is a $T$-toric variety with $\Gamma$-invariant fan determined
by a pair $(K, L)$ and the $T$-torsor $U$ determined by a triple
$(B, Q, KL)$. Two triples $(B, Q, KL)$ and $(B', Q', K'L')$ will
describe isomorphic toric varieties if and only if $K$ and $L$ are
isomorphic to $K'$ and $L'$ as $F$-algebras, (so that $T \cong
T'$), and there exist $KL$-algebra isomorphisms from $B$ to $B'$
and $Q$ to $Q'$. \label{torus}

\end{theorem}

\section{The Main Theorem}
\label{secMain}

We would now like to classify these surfaces up to isomorphism as
abstract varieties. This is less restrictive than isomorphism as
toric varieties. We will see that a del Pezzo surface is still
determined by a triple $(B, Q, KL)$, but now we will allow
$F$-algebra isomorphisms on $B$ and $Q$, i.e. algebra isomorphisms
which may not fix $KL$.

Let $S$ be a del Pezzo surface of degree 6 over $F$. Then $S$ is a
$T$-toric variety, where $T$ is the connected component of
$\Aut_F(S)$, with the $\Gamma$-action on the fan determining a
pair $(K,L)$, and the $T$-torsor $U \subset S$ determining a pair
$(B, Q, KL)$. Let $G$ be the group of connected components of
$\Aut_F(S)$. As we have shown above, $G$ is an \'{e}tale group
scheme, determined by the action of $\Gamma$ on the hexagon of
lines. In particular, $G(F) = \Aut_F(KL) \cong \Aut_F(K) \times
\Aut_F(L)$.

Consider the following action of $G(F)$ on $H^1(F,T)$: if $(g, h)
\in \Aut_L(KL) \times \Aut_K(KL)$ and $(B, Q, KL) \in H^1(F,T)$
then $g$ nontrivial sends $B$ to $B\op$, and sends the embedding
$i:KL \ra B$ to $i^{g}:KL \ra B\op$, where $i^g (z) = i(g(z))\op
\in B\op$. As $KL$ is a cyclic extension of $L$, $Q$ is a cyclic
$L$-algebra, so there is an element $l \in L^\times$ such that $Q$
is generated by $KL$ and an element $y$, subject to the relations
$y^2 = l$ and $zy = y \sigma(z)$ for every $z \in KL$, where
$\sigma \in \Aut_L(KL)$ is the nontrivial automorphism of $KL$
over $L$. Such an algebra is denoted $(KL/L, l)$. Let $h$ act on
$Q = (KL/L,L)$ by $h \cdot (KL/L, l) = (KL/L, h(l))$, and send the
embedding $KL \ra Q$ to $KL \xra{h} KL \ra (KL, h(l))$. As $l \in
L^\times$ determines $Q = (KL/L, l)$ up to multiplication by an
element of the subgroup $\NKL \subset L^\times$, and as
$\Aut_F(L)$ takes $\NKL$ to itself, we see that this action is
well defined.

The orbits of this group action can be described in terms of
$F$-algebra isomorphisms on $B$ and $Q$, as we will show below. We
will use the following proposition several times: We will need the
following proposition:
\begin{proposition}[Proposition 4.18 of \cite{KnuMerRosTig98}]
Let $(B, \tau)$ be a central simple $F$-algebra of degree $n$ with
unitary involution, and let $K$ be the center of $B$. For every
$F$-subalgebra $L$ of $B$ which is \'{e}tale of dimension $n$ over
$F$, there exists a unitary involution of $B$ fixing $L$.
\label{involution}
\end{proposition}

\begin{proposition}
$(B',Q') = g \cdot (B, Q)$ for some $g \in G(F)$ if and only if
there are $F$-automorphisms $\phi_B:B \ra B'$ and $\phi_Q:Q \ra
Q'$ such that $\phi_B|_{KL} = \phi_Q|_{KL} = g$.
\label{pairs}
\end{proposition}

\begin{proof}
Assume that $(B',Q', KL) = g \cdot (B, Q, KL)$ for some $g \in
G(F)$. Any $F$-automorphism $g$ of $KL$ can be expressed as the
composition of two automorphisms, one fixing $K$ and one fixing
$L$. So it suffices to consider the separate cases where $K$ and
$L$ are fixed by the automorphism.

We first consider the case where $g$ fixes $K$, so that $B' = B$.
By Skolem-Noether, there is a $K$-automorphism $\phi_B$ of $B$
such that $\phi_B|_{KL} = g$. Now, $Q = (KL/L, l)$, $Q' = (KL/L,
g(l))$ and as $\sigma$ and $g$ commute, $g$ extends to an
$F$-automorphism $\phi_Q$ from $Q$ to $Q'$, by sending $KL$ to
$KL$ via $g$, and $y$ to $y'$. Then $\phi_B$ and $\phi_Q$ agree on
$KL$.

Now assume that $g = \sigma$ is the non-trivial $L$-automorphism
of $KL$, so that $Q' = Q$ and $B' = B\op$. As in the previous
paragraph, by Skolem-Noether there is an $L$-automorphism $\phi_Q$
of $Q$ such that $\phi_Q|_{KL} = g $. Moreover, as $\cor_{K/F}
(B)$ is split and $L$ is an \'{e}tale cubic extension of $F$, we
know by Proposition \ref{involution} that $B$ has a unitary
involution $\tau$ which is the identity on $L$. The involution
$\tau$ defines an $F$-isomorphism $\phi_B$ from $B$ to $B\op$,
such that $\phi_B|_{KL} = \sigma = g = \phi_Q|_{KL}$.

Conversely, assume that $\phi_B :B \ra B'$ and $\phi_Q: Q \ra Q'$
are $F$-isomorphisms such that $\phi_B|_{KL} = \phi_Q|_{KL} = g
\in \Aut_F(KL)$. As in the arguments above, we will first consider
the separate cases where $g$ fix $K$ and $L$.

If $g$ fixes $L$, then $\phi_Q$ is an isomorphism of $L$-algebras,
so that $(B', Q, KL) = (B', Q', KL)$ in $H^1 (F, T)$. If we
restrict $\phi_B$ to the center of $B$, we get an $F$-isomorphism
of $K$. If $\phi|_K$ is the identity (i.e.  $g$ is trivial), then
$B$ and $B'$ are isomorphic as $K$-algebras.  If $\phi|_K$ is not
the identity, then by pre-composing $\phi$ with the isomorphism
from $B\op$ to $B$ induced by any unitary involution $\tau$ fixing
$L$ on $B$, we see that $B\op$ and $B'$ are isomorphic as
$K$-algebras. In either case, $(B', Q', KL) = (B', Q, KL) = g
\cdot (B, Q, KL)$ in $H^1 (F, T)$.

Now assume that $g$ fixes $K$, so that $\phi_B:B \ra B'$ is an
isomorphism of $K$-algebras, and then $(B', Q', KL) = (B, Q', KL)$
in $H^1 (F, T)$. If $Q = (KL/L, l)$, and if $l' = \phi_Q(l) =
g(l)$, then $Q$ is isomorphic over $L$ to $(KL/L, l')$, and so
$(B', Q', KL) = (B, Q', KL) = g \cdot (B, Q)$ in $H^1 (F,T)$.

Finally, assume that $g$ does not fix $K$ or $L$. If $\sigma$ is
the nontrivial $L$-automorphism of $KL$, $\sigma g$ does fix $K$.
Moreover, by  post-composing $\phi_B$ with the $F$-isomorphism
$\phi_{B'}: B' \ra (B')\op$ induced by any unitary involution
$\tau$ of $B'$ fixing $L$ (which exist by Proposition
\ref{involution}), we get isomorphisms $\phi_{B'} \circ \phi_B: B
\ra (B')\op$ and $\phi_Q: Q \ra Q'$ such that $(\phi_{B'} \circ
\phi_B)|_{KL} = \phi_Q|_{KL} = \sigma g$. Therefore  by our
argument in the previous paragraph, $\sigma \cdot (B', Q', KL)
=((B')\op, Q', KL) = \sigma g \cdot (B, Q, KL)$ in $H^1(F, T)$.
Acting on both sides of this equation by $\sigma$, we get $(B',
Q', KL) = g \cdot (B, Q, KL)$.
\end{proof}

The next theorem relates isomorphism classes of del Pezzo Surfaces
of degree 6 with $G(F)$-orbits of $H^1(F, T)$. We will need the
following standard result from Galois cohomology: (cf. Corollary
(28.10) of \cite{KnuMerRosTig98} or Chapter I, Section 5.5,
Corollary 2 of \cite{Ser02}.)
\begin{proposition}
Let $\Gamma$ be a profinite group, $A$ and $B$ be $\Gamma$-groups
with $A$ a normal subgroup of $B$, and set $C = B/A$. If $\beta \in
H^1(\Gamma, B)$, and $b$ a cocycle representing $\beta$, then the
elements of $H^1(\Gamma, B)$ with the same image as $\beta$ in
$H^1(\Gamma, C)$ corresponding bijectively with
$(C_b)^\Gamma$-orbits of the set $H^1 (\Gamma, A_b)$.
\label{TwistCohomology}
\end{proposition}

\begin{theorem}
The isomorphism class of $S$ corresponds to a $G(F)$-orbit of $H^1(
F, T )$.
\end{theorem}

\begin{proof}
As mentioned in Section \ref{secToricVar}, we have the following
split exact sequence of algebraic groups:
\begin{equation} 1 \ra  \widetilde{T} \ra \Aut_F(\widetilde{S}) \ra S_2 \times S_3 \ra
1,
\end{equation}
where $\widetilde{T}$ is the connected component of the identity
of $\Aut_F(\widetilde{S})$, and $S_2 \times S_3$ is the group of
connected components. This sequence induces the split exact
sequence of pointed sets:
\[1 \ra H^1(F,\widetilde{T}) \ra H^1(F,\Aut_F(\widetilde{S})) \ra H^1(F,S_2 \times S_3) \ra 1. \]
The elements of $H^1( F, \Aut_F(\widetilde{S}))$ are in a
one-to-one correspondence with the set of isomorphisms classes of
$F$-forms of $\widetilde{S}$, which by Proposition \ref{delPezzo}
are del Pezzo surfaces of degree 6. Let $\beta \in Z^1( F,
\Aut_F(\widetilde{S}))$ be a cocycle whose cohomology class is
determined by the isomorphism class of $S$, and let $\gamma$ be
the image of $\beta$ in $Z^1(F, S_2 \times S_3)$. The cocycle
$\gamma$ is determined by the action of the Galois group $\Gamma$
on $\overline{Z} \subset \overline{S}$, and induces a pair $(K,
L)$. The twist of $\widetilde{T}$ by $\gamma$ is the torus $T$,
determined by $(K, L)$ as in the sequence (\ref{eq:toric}), and
the twist of $S_2 \times S_3$ is the \'{e}tale group scheme $G$.
The result follows by Proposition \ref{TwistCohomology}.
\end{proof}

Note that as $H^1(F, \Aut_F(\widetilde{S})) \ra H^1(F, S_2 \times
S_3)$ is surjective, we see that all possible pairs $(K, L)$ are
realized by the action of $\Gamma$ on $\overline{Z}$, for $Z$
contained in some del Pezzo surface $S$ of degree 6. So let $S_1$
and $S_2$ be two del Pezzo surfaces of degree 6, and let $(B_i,
Q_i, K_iL_i)$ be an element of the the $G_i(F)$-orbit of $H^1(F,
T_i)$ determined by $S_i$, for $i = 1, 2$.  Then $S_1$ and $S_2$
induce isomorphic $\Gamma$-actions on the hexagon of lines (so
that $(K_1, L_1) \cong (K_2, L_2)$, $T_1 \cong T_2$, and $G_1
\cong G_2$; we denote these algebraic objects $(K, L)$, $T$, and
$G$, respectively). It follows from Proposition \ref{pairs} that
two pairs $(B_1,Q_1, KL)$ and $(B_2,Q_2, KL)$ are in the same
$G(F)$-orbit of $H^1(F,T)$ if and only if there are isomorphisms
$\phi_{B_1}:B_1 \ra B_2$ and $\phi_{Q_1}: Q_1 \ra Q_2$ such that
$\phi_{B_1}|_{KL} = \phi_{Q_1}|_{KL}$. We have proved the
following theorem:
\begin{theorem}
There are bijections, inverse to each other, between the following
two sets:
\begin{itemize}
\item The set of isomorphism classes of del Pezzo surfaces of
degree 6.

\item The set of triples $(B,Q, KL)$, modulo the relation: $(B, Q,
KL) \sim (B', Q', K'L')$ if there are $F$-algebra isomorphisms
$\phi_B: B \ra B'$ and $\phi_Q: Q \ra Q'$ such that $\phi_B|_{KL}
= \phi_Q|_{KL}$.
\end{itemize}
\label{main}
\end{theorem}

For the rest of this paper, $S(B,Q, KL)$ will denote the del Pezzo
surface of degree 6 determined by the triple $(B, Q, KL)$.

\begin{corollary}
The surface $S (B, Q, KL)$ contains a rational point if and only if
$B$ and $Q$ are split. \label{splitCor}
\end{corollary}

\begin{proof}
Since $S (B, Q, KL)$ is a $T$-toric variety for a two dimensional
torus $T$, $S$ has a rational point if and only if the
corresponding $T$-torsor $U$ is a trivial torsor (cf. Proposition
4 of \cite{VosKla84}). By Theorem \ref{torus}, this occurs
precisely when $B$ and $Q$ are split.
\end{proof}

\begin{remark}
If $S_0$ is a $T$-toric model (i.e. the $T$-torsor $U \subset S_0$
is trivial), then the map $H^1(F, T) \ra H^1(F, \Aut_F (S))$
induced by $T \hookrightarrow \Aut_F(S)$ takes a $T$-torsor $U$ to
the surface determined by $U$ and the $\Gamma$-action on the fan
determined by the pair $(K, L)$. Thus for any surface $S$, the
elements of $H^1(F, T)$ in the fiber of the isomorphism class of
$S$ determine the possible non-isomorphic $T$-toric structures on
$S$, where $T$ is the connected component of the identity of the
algebraic group $\Aut_F(S)$. In terms of the algebras $B$ and $Q$,
the map $H^1(F, T) \ra H^1(F, \Aut_F (S))$ forgets the
$KL$-algebra structure of $B$ and $Q$, preserving only the
$F$-algebra structure and the embedding of $KL$ into $B$ and $Q$.

The group $\Aut_F(K)$ always has order 2, but the group $\Aut_F(L)$
can have order 1, 2, 3, or 6. If $\Aut_F(L)$ has order less than 6,
then the orbit of $(B,Q, KL)$ in $H^1(F,T)$ contains at most 6
elements. It $\Aut_F(L)$ has order 6, then $L= F^3$ is not a field,
and thus $B$ is necessarily split. If $B$ is split, then the pair
$(B,Q, KL) \in H^1(F,T)$ is fixed by the subgroup $\Aut_F(K)$ of
$G(F)$, and so again the $G(F)$-orbit of $(B,Q, KL)$ in $H^1(F,T)$
has at most 6 elements. Thus for a del Pezzo surface $S$ of degree
6, there at most 6 non-isomorphic $T$-toric structures on $S$.
\end{remark}


\begin{remark}
We would like to relate this characterization of del Pezzo surfaces
of degree 6 by triples $(B,Q, KL)$ with the characterization by
triples $(B, \tau, L)$ found in \cite{ColKarMer07}. A triple $(B,
\tau, L)$ is an Azumaya $K$-algebra $B$ of rank 9, a unitary
involution $\tau$ on $B$, and a cubic \'{e}tale $F$-algebra $L$ such
that $L \subset \Sym(B, \tau)$. Two triples $(B, \tau, L)$ and $(B',
\tau', L')$ are isomorphic if there is an $F$-algebra isomorphism
$\phi: B \ra B'$ such that $\tau'\phi = \phi\tau$ and $\phi(L) =
L'$.

So let $(B, Q, KL)$ be a triple as in Theorem \ref{main}. Then $B$
is an Azumaya $K$-algebra of rank 9, and is classified up to
isomorphism as an $F$-algebra. As $B_L$ is split, so $B$ contains
$KL$, and hence $L$, as a subalgebra. Since $\cor_{K/F} (B)$ is
split, we know that $B$ has a unitary involution. Moreover, since
$L$ is an \'{e}tale cubic extension of $F$ contained in $B$, there
is some unitary involution $\tau$ such that $L \subset \Sym (B,
\tau)$, by Proposition \ref{involution}. So the $B$ and $L$ in our
characterization match with the $B$ and $L$ described in
\cite{ColKarMer07}. The rest of the remark seeks to relate $Q$ and
the involution $\tau$. That is, we want to classify all triples $(B,
\tau, L)$ with $B$ and $L$ fixed. This should correspond to fixing
$K$, $L$, and $B$, and trying to determine all possible $Q$.

As $K/F$ and $KL/L$ are cyclic,
\[\Br(K/F) \cong F^\times/\NK
\] and
\[\Br(KL/L) \cong L^\times/\NKL. \]
The restriction homomorphism $\res_{L/F}: \Br(F) \ra \Br(L)$ sends
the subgroup $\Br(K/F)$ to $\Br(KL/L)$. As $K$ and $L$ have
coprime degrees, $\res_{L/F}|_{\Br(K/F)}: \Br(K/F) \ra \Br(KL/L)$
is injective, and the cyclic algebra $Q$ corresponds to an element
of $\Br(KL/L) / \res_{L/F}(\Br(K/F) )$, i.e. an element of
\begin{align*}
&\ L^\times/ \NKL \Big / \res_{L/F} \Big( F^\times / \NK \Big )  \\
=&\ L^\times/ \NKL \Big / F^\times \NKL / \NKL  \\
\cong &\ L^\times/ F^\times \NKL.
\end{align*}

If $\tau$ is a unitary involution on $B$ which is the identity on
$L$, and $u \in L^\times$, then $\tau_u : = \Int(u) \circ \tau$ is
also a unitary involution on $B$ fixing $L$. Moreover, $\tau_u$ is
conjugate to $\tau_v$ if $uv^{-1} \in F^\times\NKL$. Thus, after a
choice of a particular involution $\tau$, we have a morphism of
pointed sets from $L^\times/ F^\times \NKL$ to the set of
conjugacy classes of unitary involutions of $B$ which are the
identity on $L$, sending $u$ to $\tau_u$. By Corollary 19.3 of
\cite{KnuMerRosTig98}, this map is a surjection. So we have a
surjective map from $L^\times/ F^\times \NKL$ to the set of
isomorphism classes of triples $(B, \tau, L)$.

Theorem \ref{main} should say that the fibers of this surjection
should be the orbit of the group $\Aut_F(L)$ in $L^\times/
F^\times \NKL$, corresponding to the $\Aut_F(L)$-orbit of $(B,
Q)$. However, to make this statement correct, we need to choose a
particular involution $\tau$ on $B$. It is not clear in general
what this involution should be. The involution $\tau$ should be
chosen so that the surface $S (B, \tau, L)$ should correspond to
the pair $(B, M_2(L))$. In particular, if $B = M_3(K)$ is split,
the surface described by the triple $(M_3(K), \tau, L)$ should
have a rational point, by Corollary \ref{splitCor}. The next
remark constructs the involution in this case.
\end{remark}

\begin{remark}
Given $K$ and $L$, we will find a triple $(B, \tau, L)$, (i.e.  a
central simple algebra of degree 3 over $K$ with an involution
$\tau$ such that $L \subset \Sym (B, \tau)$,) so that the
corresponding del Pezzo surface $S(B, \tau, L)$ constructed in
\cite{ColKarMer07} has a rational point.

As $KL$ is a three dimensional vector space over $K$, $B :=
\End_{K}(KL)$ is a Azumaya $K$-algebra of rank 9. Left
multiplication by an element of $KL$ determines an embedding of
$KL$ into $B$. If $\sigma$ is the nontrivial $L$-automorphism of
$KL$, $h(x,y) = \Tr_{KL/K} (\sigma(x) y )$ defines a hermitian
form on $KL$. This hermitian form on $KL$ induces an involution of
the second kind $\tau$ on $B$, such that $L \subset \Sym (B,
\tau)$. So we have a triple $(B, \tau, L)$. Let $S$ denote the
corresponding del Pezzo surface of 6, constructed in
\cite{ColKarMer07}.

I claim that $S$ contains a rational point. According to
\cite{ColKarMer07}, it suffices to show that there is a right
ideal $I$ of $B$ of reduced dimension 1 such that $(I \cdot
\tau(I)) \cap \Sym (B, \tau) \subset F \oplus L^\perp$, where
$L^\perp = \{x \in \Sym (B, \tau) | \Trd (l x) = 0, \ \mathrm{for\
all}\ l \in L \}$.

Let $W = \spann_K(1) \subset LK$, so that $W$ is a one dimensional
$K$-subspace of $KL$. and then $I = \Hom_{K}(KL,W)$ is a right
ideal of $B$ of reduced dimension 1, generated by the linear map
$t = \Tr_{KL/K}: KL \ra K \hookrightarrow KL$. We want to show
that $t \in \Sym (B,\tau)$. First, note that for any $x \in KL$,
$\Tr_{KL/K} (\sigma(x)) = \sigma( \Tr_{LK/L}(x) )$. If $x, y \in
KL$,
\begin{align*}
h(x,t(y)) & = \Tr_{KL/K}(\sigma(x)\Tr_{KL/K}(y)) \\
& =  \Tr_{KL/K}(\sigma(x))\Tr_{KL/K}(y)  \\
& =  \Tr_{KL/K}(\sigma(\Tr_{KL/K}(x))y) \\
& =  h(t(x),y).
\end{align*}
So $t \in \Sym(B, \tau)$, which implies that $\tau(I)$ is a left
ideal of $B$, also generated by $t$. Thus, $I \cdot \tau(I) = t B
t$.

In order to prove $I \cdot \tau(I) \cap \Sym (B, \tau) \subset F
\oplus L^\perp$, it suffices to consider the case where $L = F^3$
is split. So we can choose a basis $e_1, e_2, e_3$ of idempotents
for $KL$ over $K$.  In this basis, $\tau$ is the standard adjoint
involution, $\Sym (B, \tau)$ is the set of hermitian matrices, and
$F \oplus L^\perp$ is the set of hermitian matrices where the
diagonal entries agree. Moreover, $t$ is the matrix with ones in
every entry, so $t \in F \oplus L^\perp$. A direct calculation
shows $I \cdot \tau(I) = \spann_K(t)$, and hence  $(I \cdot
\tau(I)) \cap \Sym (B, \tau) = \spann_F(t) \subset F \oplus
L^\perp$.
\end{remark}

\section{$K_0$ of del Pezzo Surfaces}
\label{secK_0}

Let $S = S(B, Q, KL)$ be a del Pezzo surface of degree 6, a
$T$-toric variety for a two dimensional torus $T$. Let $Z \subset
S$ be the closed variety such that $\overline{Z}$ is the union of
six lines $l_0$, $l_1$, $l_2$, $m_0$, $m_1$, and $m_2$. Recall the
exact sequence (\ref{PicardModule}) from Section
\ref{secToricVar}, where $\Z[KL/F]$ is the lattice of connected
components of $\overline{Z} \subset \overline{S}$, the lines $l_i$
and $m_i$, and the homomorphism $\Z[KL/F] \ra \Pic(\overline{S})$
sends each line to the corresponding invertible sheaf on
$\overline{S}$. From the exact sequence (\ref{GammaModule}), we
see that $\widehat{T}$ is the subgroup of $\Z[KL/F]$ generated by
$l_0 - l_1 - (m_0 - m_1)$, $l_0 - l_2 - (m_0 - m_2)$, and $l_1 -
l_2 - (m_1 - m_2)$. Note that any one of these 3 generators can be
expressed as a linear combination of the other 2. So $\Pic
(\overline{S})$ is generated by the invertible sheaves
$\cL(-l_i)$, $\cL (-m_j)$, and we have that the invertible sheaves
$\cL (-l_i - m_j)$ and $\cL (-l_j - m_i)$ are isomorphic for $i, j
= 0, 1, 2$. 

There is another way to recover these generators and relations,
which does not depend on the theory of toric varieties. Recall
that there is a morphism $p_1:\overline{S} \ra \P^2$, obtained by
blowing down the lines $m_0$, $m_1$, and $m_2$. If $x_0$, $x_1$,
$x_2$ are the homogeneous coordinates of $\P^2$ and $D_i = \{x_i =
0\}$ for $i = 0, 1, 2$, then $D_0$, $D_1$, and $D_2$ are all
linearly equivalent divisors on $\P^2$, and thus their strict
transforms $m_1 + l_0 + m_2$, $m_0 + l_1 + m_2$, and $m_0 + l_2 +
m_1$ are all linearly equivalent divisors on $\overline{S}$.
Therefore the corresponding invertible sheaves $\cL (-m_1 - l_0 -
m_2)$, $\cL (-m_0 - l_1 - m_2)$, and $\cL (-m_0 - l_2 - m_1)$ on
$\overline{S}$ are isomorphic. From this we can conclude that
$\Pic(\overline{S})$ is generated by the invertible sheaves
$\cL(-m_0)$, $\cL(-m_1)$, $\cL(-m_2)$, $\cL (-m_1 - l_0 - m_2)$,
$\cL (-m_0 - l_1 - m_2)$, and $\cL (-m_0 - l_2 - m_1)$, and we
have that $\cL (-l_i - m_j - l_k)$ and $\cL (-l_j - m_i - l_k)$
are isomorphic for any $i, j = 0, 1, 2$. This presentation is
equivalent to that in the previous paragraph. Similarly, this
presentation can be obtained by considering the morphism $p_2:
\overline{S} \ra \P^2$ obtained by blowing down the lines $l_0$
$l_1$, and $l_2$.

We define the following locally free sheaves on $\overline{S}$:
\begin{align*}
\cI_1 &= \cL (-m_1 - l_0 - m_2) \oplus \cL (-m_0 - l_1 - m_2)
\oplus \cL(-m_0 - l_2 - m_1) \\
\cI_2 & = \cL (-l_1 - m_0 - l_2) \oplus \cL (-l_0 - m_1 - l_2)
\oplus \cL (-l_0 - m_2 - l_1) \\
\cJ_1 &= \cL (-l_0 - m_1) \oplus \cL (-l_1 - m_0) \\
\cJ_2 &= \cL (-l_0 - m_2) \oplus \cL (-l_2 - m_0) \\
\cJ_3 &= \cL  (-l_1 - m_2) \oplus \cL (-l_2 - m_1).
\end{align*}
The $\Gamma$-action on the hexagon of lines induces an action on the
locally free sheaves $\cI_1 \oplus \cI_2$ and $\cJ_1 \oplus \cJ_2
\oplus \cJ_3$, compatible with the action on $\overline{S}$. Therefore
$\cI_1 \oplus \cI_2$ and $\cJ_1 \oplus \cJ_2 \oplus \cJ_3$ descend to
sheaves $\cI$ and $\cJ$ on $S$.

We will consider the following endomorphism rings: $B' =
\End_{\cO_S} (\cI) \op$, and $Q' = \End_{\cO_S} (\cI)\op$. As $S$
is projective, $\End_{\cO_S}(\cO_S)\op = F$, and since $\cI$ and
$\cJ$ are $\cO_S$-modules, it follows that $B'$ and $Q'$ are
$F$-algebras. For $i$, $j$, and $k$ not equal,
$\End_{\cO_{\overline{S}}}(\cL (-m_i - l_j - m_k)) =
\End_{\cO_{\overline{S}}}(\cL (-l_i - m_j - l_k)) = \overline{F}$,
so we see that $\overline{F}^6$ embeds diagonally into
$\End_{\cO_{\overline{S}}} (\cI_1 \oplus \cI_2)$. Moreover, since
$\cL (-l_i - m_j)$ and $\cL (-l_j - m_i)$ are isomorphic for any
$i$, $j$, $\cI_1 \oplus \cI_2 = (\cL (-m_1 - l_0 - m_2) \oplus \cL
(-l_1 - m_0 - l_2)) \tens_{\overline{F}} V$, where $V$ is an
$\overline{F}$-vector space of dimension 3. An element of
$\Hom_{\cO_{\overline{S}}}(\cL (-m_1 - l_0 - m_2), \cL (-l_1 - m_0
- l_2) )$ is given by a global section of $\cL (m_2 - l_2)$. Any
non-zero global section of $\cL (m_2 - l_2)$ would give a function
defined on a neighborhood of $l_2 \subset \overline{S}$ with
vanishing set $l_2$. Blowing down the lines $l_i$, this function
would then correspond to a function defined on an open subset of
$\P^2$ with vanishing set a point, which is impossible, since a
point is a codimension 2 subvariety of $\P^2$. Thus $\cL (m_2 -
l_2)$ has no nonzero global sections. Similarly
$\Hom_{\cO_{\overline{S}}} ( \cL (-l_1 - m_0 - l_2), \cL (-m_1 -
l_0 - m_2) ) = 0$, and so $\End_{\cO_{\overline{S}}}(\cI_1 \oplus
\cI_2) = \End_{\cO_{\overline{S}}}(\cL (-m_i - l_j - m_k) \times
\cL (-l_i - m_j - l_k)) \tens_{\overline{F}} \End_{\overline{F}}
(V) = \End_{\overline{F}^2}(V_{\overline{F}^2})$. So the center of
$\End_{\cO_{\overline{S}}} (\cI_1 \oplus \cI_2)$ is a copy of
$\overline{F}^2$, contained in $\overline{F}^6$. This chain
$\overline{F}^2 \subset \overline{F}^6 \subset
\End_{\cO_{\overline{S}}}(\cI_1 \oplus \cI_2)$ descends to $K
\subset KL \subset \End_{\cO_S}(\cI)$, with $K$ the center of
$\End_{\cO_S}(\cI)$.

Now let $E$ be any separable field extension of $F$ over which the
lines $l_i$, $m_j$ are defined. This is equivalent to $E$ splitting
both $K$ and $L$. the above arguments show that $\End_{\cO_S}(\cI)
\tens_F E \approx M_3(E^2)$, where $E^2 \approx K \tens_F
E$. Therefore, we conclude that $B'$ is an Azumaya $K$-algebra of rank
9 which contains $KL$ as a subalgebra. A similar argument shows that
$Q'$ is an Azumaya $L$-algebra of rank 4 which also contains a copy of
$KL$.

\begin{theorem}
$B' = \End_{\cO_S}(\cI)\op$ and  $B$ are isomorphic as $K$-algebras.
Similarly, $Q' = \End_{\cO_S}(\cJ)\op$ and $Q$ are isomorphic as
$L$-algebras.
\end{theorem}

\begin{proof}
Let $\widetilde{S}$ be as in Proposition \ref{delPezzo}, and let
$\widetilde{\cI}$ and $\widetilde{\cJ}$ be the sheaves associated to
$\widetilde{S}$ as above.  Twisting $\widetilde{\cI}$ and
$\widetilde{\cJ}$ by the $\Gamma$-action corresponding to the pair
$(K,L)$, we get sheaves $\cI_0$ and $\cJ_0$ associated to the
$T$-toric model $S_0$. Let $B'_0 = \End_{\cO_S}(\cI_0)\op$ and $Q'_0
= \End_{\cO_S}(\cJ_0)\op$.  The embeddings of $KL$ into $B'_0$ and
$Q'_0$ described above induce the following commutative diagrams
(cf. (\ref{eq:toric2})):

\[ \xymatrix{
1 \ar[r] & R^{(1)}_{K/F} (\gm) \ar[r] \ar[d] & \mathbf{G}_L \ar[r] \ar[d] & T \ar[r] \ar[d] & 1 \\
1 \ar[r] & R_{K/F} (\gm) \ar[r] & R_{K/F} (\gGL(B_0')) \ar[r] &
R_{K/F} (\gPGL(B_0')) \ar[r] & 1, }\] and
\[ \xymatrix{
1 \ar[r] & R^{(1)}_{L/F} (\gm) \ar[r] \ar[d] & \mathbf{G}_K \ar[r] \ar[d] & T \ar[r] \ar[d] & 1 \\
1 \ar[r] & R_{L/F} (\gm) \ar[r] & R_{L/F} (\gGL(Q_0')) \ar[r] &
R_{L/F} (\gPGL(Q_0')) \ar[r] & 1. }\]

The first diagram induces the following commutative diagram of
cohomology sets:
\[\xymatrix{
& H^1(F, T) \ar[r] \ar[d] & \Ker (\cor_{K/F}: \Br(K) \ra \Br (F)) \ar[d] \\
1 \ar[r] & H^1(K, \gPGL(B_0')) \ar[r] & \Br(K). }\]

The left vertical arrow sends the triple $(B, Q, KL)$ to the
endomorphism ring $B' = \End(\cI)\op$, where $\cI$ is the sheaf
associated to the surface $S(B,Q, KL)$ constructed above. The
upper horizontal arrow sends the triple $(B,Q, KL)$ to the class
of $[B] \in \Br(K)$, which lands in the subgroup of elements of
trivial norm. The right vertical map is the inclusion
homomorphism. The lower horizontal arrow sends a $K$-algebra to
its corresponding element in $\Br(K)$. By the commutativity of the
diagram, we see that $[B] = [B']$ in $\Br(K)$. But these algebras
have the same rank, so they must be isomorphic as $K$-algebras. A
similar argument shows that $Q' = \End_{\cO_S}(\cJ)\op$ and $Q$
are isomorphic as $L$-algebras.
\end{proof}

Let $A$ be the separable $F$-algebra $F \times B \times Q$, so
that $\mathbf{P} (A) = \mathbf{P} (F) \times \mathbf{P} (B) \times
\mathbf{P} (Q)$. Define the exact functors $u_F$ from
$\mathbf{P}(F)$ to $\mathbf{P}(S)$ by $M_1 \mapsto \cO_S \tens_F
M_1$, $u_B$ from $\mathbf{P} (B)$ to $\mathbf{P} (S)$ by $M_2
\mapsto \cI \tens_{B} M_2$, and $u_Q$ from $\mathbf{P} (Q)$ to
$\mathbf{P}(S)$ by $M_3 \mapsto \cJ \tens_{Q} M_3$. If we set $\cP
= \cO_S \oplus \cI \oplus \cJ$, then the respective right actions
of $F$, $B$, and $Q$ on $\cO_S$, $\cI$, and $\cJ$ combine to give
a right action of $A = F \times B \times Q$ on $\cP$. Therefore,
we can define an exact functor from $\mathbf{P}(A)$ to
$\mathbf{P}(S)$ by sending $M$ to $\cP \tens_{A} M$. This exact
functor induces a homomorphism:
\[ \phi: K_0(A) \ra K_0 (S).\]

More generally, if $Y$ is any $F$-variety, then we have an exact
functor from $\mathbf{P} (Y ; A)$ to $\mathbf{P} (Y \times S)$,
sending $M$ to $p_2^*(\cP) \tens_{\cO_{Y \times S} \tens_F A}
p_1^*(M)$, where $p_1:Y \times S \ra Y$ and $p_2:Y \times S \ra S$
are the projection morphisms. This induces a homomorphism $\phi_Y:
K_0(Y;A) \ra K_0(Y \times S)$. Furthermore, if $E$ is any field
extension of $F$, then $\phi$ naturally extends to a homomorphism
$\phi_E: K_0(A_E) \ra K_0 (S_E)$.

\begin{theorem}
$\phi: K_0(A) \ra K_0(S)$ is an isomorphism. \label{zero}
\end{theorem}

We will prove this in several stages. Let us first consider the
case where $F$ is separably closed. By Proposition \ref{delPezzo},
$S$ is isomorphic to the blow up of the projective plane at the 3
non-collinear points $[1:0:0]$, $[0:1:0]$, and $[0:0:1]$. Recall
that we have the filtration $0 = K_0 (S)^{(3)} \subset K_0
(S)^{(2)} \subset K_0 (S)^{(1)} \subset K_0 (S)^{(0)} = K_0 (S)$
by codimension of support, and homomorphisms $CH^i(S) \ra K_0
(S)^{(i/i+1)}$, which send the class of a subvariety $V$ to the
equivalence class $[\cO_V]$. These homomorphisms are isomorphisms
for $i = 0, 1, 2$. So by Proposition \ref{delPezzo}, if $P$ is a
rational point of $S$, $K_0 (S)$ is generated by $[\cO_S]$,
$[\cO_{l_0}]$, $[\cO_{l_1}]$, $[\cO_{l_2}]$, $[\cO_{m_0}]$,
$[\cO_{m_1}]$, $[\cO_{m_2}]$, and $[\cO_P]$. Moreover, as $\CH^0
(S)$, $\CH^1 (S)$, and $\CH^2 (S)$ are free abelian groups with
ranks 1, 4, and 1, respectively, $K_0(S)$ is free abelian with
rank 6. Since $F$ is separably closed, $K$, $L$, $B$, and $Q$ are
split, and thus $K_0(A)$ is also free abelian of rank 6. Therefore
$\phi$ will be an isomorphism provided it is surjective. So it
suffices to show that $[\cO_S]$, $[\cO_{l_0}]$, $[\cO_{l_1}]$,
$[\cO_{l_2}]$, $[\cO_{m_0}]$, $[\cO_{m_1}]$, $[\cO_{m_2}]$, and
$[\cO_P]$ are in the image of $\phi$.

Clearly, $[\cO_S] = [\cO_S \tens_F F]$ is in the image of $\phi$.
As $F$ is separably closed, $\cI = (\cL (-m_1 - l_0 - m_2) \oplus
\cL (-l_1 - m_0 - l_2)) \tens_F V$, where $V$ is an $F$-vector
space of dimension 3, $K = F \times F \cong \End_{\cO_S}(\cL (-m_1
- l_0 - m_2)) \times \End_{\cO_S} ( \cL (-l_1 - m_0 - l_2) )$, and
$\End_{\cO_S}(\cI)\op \cong \End_{F^2}(V_{F^2})$.

Now $\Hom_{F^2}(V_{F^2},F \times 0)$ is a right
$\End_{\cO_S}(\cI)$-module, and thus a left $A$-module, where the
$F$ and $Q$ component of $A = F \times B \times Q$ act trivially.
Therefore,
\begin{align*}
\phi \Big(\Hom_{F^2}(V_{F^2},F \times 0)\Big) &= \Big[
\cI \tens_B \Hom_{F^2}(V_{F^2}, F \times 0) \Big] \\
& = \Bigg[ \Big(\cL (-m_1 - l_0 - m_2) \oplus \cL (-l_1 - m_0 -l_2)
\Big) \\ & \tens_{F^2} (F \times 0)\Bigg] \\
& = [\cL (-m_1 - l_0 - m_2)],
\end{align*}
where we use Morita equivalence in the second line. A mirror
argument shows that $\phi\Big(\Hom_{F^2}(V_{F^2}, 0 \times F)\Big)
= [\cL (-l_1 - m_0 - l_2)]$, and a similar argument applied to
$\cJ$ and $Q$ shows that $[\cL (-l_0 - m_1)]$, $[\cL (-l_0
  - m_2)]$, and $[\cL (-l_1 - m_2)]$ are in the image of $\phi$.

Now let $i, j \in \{0,1,2\}$ and not equal. By Proposition
\ref{delPezzo}, the lines $l_i$ and $m_j$ have intersection a
rational point $P$ of $S$, with multiplicity 1, the lines $m_i$
and $m_j$ are skew, and the lines $l_i$ and $l_j$ are skew. Thus
we have the following resolutions of $\cO_P$ and $\cO_S$:

\begin{equation*}
0 \ra \cL(-l_i - m_j) \xra{\begin{pmatrix}\tens\cL(m_j) \\ \tens
\cL (l_i) \end{pmatrix}} \cL(-l_i) \oplus \cL (-m_j)
\xra{(\tens\cL(l_i), -\tens\cL(m_j))} \cO_S \ra \cO_P \ra 0,
\end{equation*}

\begin{equation*}
0 \ra \cL(-m_i - m_j) \xra{\begin{pmatrix}\tens \cL(m_i) \\ \tens
\cL(m_j) \end{pmatrix}} \cL(-m_i) \oplus \cL (-m_j) \xra{(\tens
\cL(m_i), -\tens \cL(m_j))} \cO_S \ra 0,
\end{equation*}
and
\begin{equation*}
0 \ra \cL(-l_i - l_j) \xra{\begin{pmatrix}\tens \cL(l_i) \\ \tens
\cL(l_j) \end{pmatrix}} \cL(-l_i) \oplus \cL (-l_j) \xra{(\tens
\cL(l_i), -\tens \cL(l_j))} \cO_S \ra 0.
\end{equation*}
In addition, when $D = l_i$ or $m_i$, we have the standard
resolution
\begin{equation*}0 \ra \cL(-D) \xra{\tens \cL(D)} \cO_S \ra \cO_D \ra 0.
\end{equation*}

So $0 = [\cO_S] - [\cL(-m_i)] - [\cL(-m_j)] + [\cL(-m_i - m_j)]$ in
$K_0(S)$. If we take $k \in \{0,1,2\}$ not equal to $i$ or $j$ and
multiply this equation by $[\cL(-l_k)]$, we see that \[0 =
[\cL(-l_k)] - [\cL(-l_k - m_i)] - [\cL(-l_k - m_j)] + [\cL(-m_i -
l_k -  m_j)].\] Therefore,
\begin{align*}[\cO_{l_k}] & =
[\cO_S] - [\cL(-l_k)] \\ &= [\cO_S] -
[\cL(-l_k - m_i)] - [\cL(-l_k - m_j)] + [\cL(-m_i - l_j - m_k)]
\end{align*}
is in the image of $\phi$. The same argument with $l$ and $m$
interchanged shows that $[\cO_{m_k}]$ is in the image of $\phi$ for
$k \in \{0, 1, 2\}$. Finally, $[\cO_P] = [\cO_S] - [\cL(-l_0)] -
[\cL(-m_1)] + [\cL(-l_0 - m_1)]$ is in the image of $\phi$, and thus
$\phi$ is surjective when $F$ is separably closed.

\begin{proposition}
$\phi: K_0(A) \ra K_0(S)$ is an isomorphism if $B$ and $Q$ are split.
\label{splitThm}
\end{proposition}
\begin{proof}
By the preceding argument, $\phi_{\overline{F}}:K_0 (\overline{A})
\ra K_0 (\overline{S})$ is an isomorphism. Moreover,
$\phi_{\overline{F}}$ commutes with the action of $\Gamma$ on both
$K_0(\overline{A})$ and $K_0(\overline{S})$, and thus it descends
to an isomorphism on the $\Gamma$-invariant subgroups. Therefore,
we have the following commutative diagram:
\[
\xymatrix{
K_0(A) \ar[d]  \ar[r]^\phi  & K_0(S)   \ar[d]\\
K_0(\overline{A})^\Gamma \ar[r]^{\phi_{\overline{F}}} &
K_0(\overline{S})^\Gamma.}
\]
As $\phi_{\overline{F}}$ and the left vertical map $K_0(A) \ra
K_0(\overline{A})^\Gamma$ are isomorphisms, $\phi$ must be
injective. Moreover, if the right vertical map is injective, then
$\phi$ is surjective, and hence an isomorphism. So it suffices to
show that $K_0(S) \ra K_0(\overline{S})^\Gamma$ is injective.

To see this, note that the rank and wedge homomorphisms $\rank:
K_0(\overline{S}) \ra \Z$ and $\wedge: K_0(\overline{S})^{(1)} \ra
\Pic(\overline{S})$ commute with the action of $\Gamma$. Thus we
have the following short exact sequences of $\Gamma$-modules:
\[0 \ra K_0(\overline{S})^{(2)} \ra K_0(\overline{S})^{(1)}
\xra{\wedge} \Pic(\overline{S}) \ra 0 \] and
\[0 \ra K_0(\overline{S})^{(1)} \ra K_0(\overline{S})
\xra{\rank} \Z \ra 0. \] These sequences of $\Gamma$-modules induce
the following long exact sequences:
\[0 \ra (K_0(\overline{S})^{(2)})^\Gamma \ra (K_0(\overline{S})^{(1)})^\Gamma \xra{\wedge}
\Pic(\overline{S})^\Gamma \ra
H^1(F,K_0(\overline{S})^{(2)})\] and
\[ 0 \ra (K_0(\overline{S})^{(1)})^\Gamma \ra (K_0(\overline{S}))^\Gamma \xra{\rank}
\Z \ra H^1(F,K_0(\overline{S})^{(1)}).\]

The map $K_0(S) \ra K_0(\overline{S})^\Gamma$ induces the following
commutative diagrams:
\[\xymatrix{
0 \ar[r] & K_0(S)^{(2)} \ar[r] \ar[d] & K_0(S)^{(1)} \ar[r]^\wedge \ar[d] & \Pic(S) \ar[r] \ar[d] & 0 \\
0 \ar[r] & (K_0(\overline{S})^{(2)})^\Gamma \ar[r] &
(K_0(\overline{S})^{(1)})^\Gamma \ar[r]^\wedge &
\Pic(\overline{S})^\Gamma  \ar[r] & H^1(F,K_0(\overline{S})^{(2)}),}
\]
and
\[\xymatrix{
0 \ar[r] & K_0(S)^{(1)} \ar[r] \ar[d] & K_0(S)
\ar[r]^{\rank} \ar[d] & \Z \ar[r] \ar[d]^= & 0 \\
0 \ar[r] & (K_0(\overline{S})^{(1)})^\Gamma \ar[r] &
K_0(\overline{S})^\Gamma \ar[r]^{\rank} & \Z \ar[r] &
H^1(F,K_0(\overline{S})^{(1)}). } \]

As $B$ and $Q$ are split, $S$ has a rational point by Corollary
\ref{splitCor}. Thus the homomorphism $K_0(S)^2 \ra ( K_0 (
\overline{S} )^{(2)} )^\Gamma$ is a surjective homomorphism of
free abelian groups of rank 1, and therefore an isomorphism.
Moreover, the homomorphism $\Pic (S) \ra
\Pic(\overline{S})^\Gamma$ is injective. Thus, by applying the
Snake Lemma to the first diagram and then to the second, we see
that $K_0(S) \ra (K_0 ( \overline{S} ) )^\Gamma$ is injective.
\end{proof}

\begin{remark}
If $P$ and $P'$ are rational points of $\overline{S}$, they define
equal classes in $\CH^2(\overline{S})$ by Proposition
\ref{delPezzo}, and hence $[\cO_P] = [\cO_{P'}]$ in $K_0(
\overline{S} )^{(2)}$. So $K_0( \overline{S} )^{(2)}$ is generated
by the $\Gamma$-invariant element $[\cO_P]$, and thus is a trivial
$\Gamma$-module. Similarly, our arguments above show that the set
\{$[\cO_{l_0} \oplus \cO_{m_1}]$, $[\cO_{l_0} \oplus \cO_{m_2}]$,
$[\cO_{l_1} \oplus \cO_{m_2}]$, $[\cO_{m_1} \oplus \cO_{l_0}
\oplus \cO_{m_2}]$, $[\cO_{l_1} \oplus \cO_{m_0} \oplus
\cO_{l_2}]$\} is a $\Gamma$-invariant basis of
$K_0(\overline{S})^{(1)}$, and thus $K_0(\overline{S})^{(1)}$ is a
permutation module. It follows that
$H^1(F,K_0(\overline{S})^{(2)}) = H^1(F,K_0(\overline{S})^{(1)}) =
0.$
\end{remark}

We will need the following proposition (cf. Proposition 6.1 of
\cite{MerPan97}).
\begin{proposition}
If $Y$ is a variety such that the homomorphism $\phi_{F(y)}: K_0
(A_{F(y)}) \ra K_0 (S_{F(y)})$ is an isomorphism for every $y \in
Y$, then $\phi_Y: K_0 (Y ; A) \ra K_0 (S \times Y)$ is surjective.
\label{surjection}
\end{proposition}

\begin{proof}
We do this by double induction on the dimension of $Y$ and the
number of irreducible components of $Y$.

If $Y$ has a proper irreducible component $Y'$, with complement
$U$, we have the following localization exact sequence (cf.
\cite{Qui72}):
\[
\xymatrix{K_0(Y'; A) \ar[r] \ar[d]^{\phi_{Y'}} & K_0(Y; A)
\ar[r]\ar[d]^{\phi_Y}
& K_0(U; A) \ar[r] \ar[d]^{\phi_U} & 0 \\
K_0(Y' \times S) \ar[r] & K_0(Y \times S) \ar[r] & K_0(U \times S)
\ar[r]  & 0 }
\]
By our inductive assumption, the vertical maps on the right and on
the left are surjective. This implies that the middle vertical map
is surjective as well. So we may assume that $Y$ is irreducible.
If $Y$ is not reduced, then we the natural $Y_{\red} \ra Y$
induces the commutative diagram
\[ \xymatrix{
K_0(Y_{\red}; A)  \ar[r] \ar[d]^{\phi_{Y_{\red}}} & K_0(Y; A) \ar[d]^{\phi_Y} \\
K_0(Y_{\red} \times S) \ar[r]   & K_0(Y \times S), }\] where the
horizontal arrows are isomorphisms. Thus we may also assume that
$Y$ is reduced.

Now let $x \in K_0 (Y \times S)$. By assumption, $\phi_{F(Y)}:
K_0(A_{F(Y)}) \ra K_0 (S_{F(Y)})$ is an isomorphism, and hence
there exists an open set $U$ in $Y$ such the image of $x$ in $K_0
(U \times S)$ is in the image of $\phi_U$. We again consider the
localization exact sequence:
\[
\xymatrix{K_0(Z; A) \ar[r] \ar[d]^{\phi_Z} & K_0(Y; A)
\ar[r]\ar[d]^{\phi_Y}
& K_0(U; A)  \ar[r] \ar[d]^{\phi_U} & 0 \\
K_0(Z \times S) \ar[r] & K_0(Y \times S) \ar[r] & K_0(U \times S)
\ar[r]  & 0, }
\]
where $Z$ is the complement of $U$ in $Y$. By assumption, $Z$ has a
strictly smaller dimension than $Y$, and so by our inductive
hypothesis, $\phi_Z$ is surjective. A standard diagram chase shows
that $x \in \Im(\phi_Y)$.
\end{proof}

\begin{proof}[Proof of Theorem \ref{zero}]
Let $SB(B)$ be the Severi-Brauer $K$-variety associated to $B$,
$SB(Q)$ be the Severi-Brauer $L$-variety associated to $Q$, and $Y =
R_{K/F} (SB(B)) \times R_{L/F} (SB(Q))$ be the product of the
restriction of scalars of both varieties. Then for any field
extension $E$ of $F$, $Y(E)$ is nonempty if and only if $SB(B)(K
\tens_F E)$ and $SB(Q)(L \tens_F E)$ are nonempty if and only if
$B_E = B \tens_K (K \tens_F E)$ and $Q_E = Q \tens_L (L \tens_F E)$
are split.

The projection $p: Y \ra \Spec (F)$ induce the following diagram:
\[
\xymatrix{ K_0(A)     \ar[r]^\phi \ar@<1ex>[d]^{p^*} & K_0(S) \ar@<1ex>[d]^{(p_Y)^*} \\
           K_0(Y ; A) \ar[r]^{\phi_Y} \ar[u]^{p_*}   & K_0(Y \times S), \ar[u]^{(p_Y)_*}}
\]
where $p_Y: Y \times S \ra S$ is the projection induced by $p$.
Both squares commute, and $p_*p^*$ is the identity homomorphism,
as $Y$ is a geometrically rational variety. For every $y \in Y$,
$Y(F(y)) \neq \emptyset$, so $B_{F(y)}$ and $Q_{F(y)}$ are split,
and thus $\phi_{F(y)}: K_0(A_{F(y)}) \ra K_0 (S_{F(y)})$ is an
isomorphism by Proposition \ref{splitThm}. So by Proposition
\ref{surjection}, $\phi_Y: K_0 (Y;A) \ra K_0 (Y \times S)$ is
surjective. A diagram chase shows the top horizontal map $\phi$ is
also surjective.

Now let $E$ be any field extension of $F$ such that $S(E) \neq
\emptyset$. Then, $B_E$ and $Q_E$ are split, and so $\phi_{E}: K_0
(A_E) \ra K_0 (S_E)$ is an isomorphism, again by Proposition
\ref{splitThm}. The homomorphisms $\phi$ and $\phi_{E}$ fit into
the following commutative diagram:
\[
\xymatrix{ K_0(A)     \ar[r]^\phi \ar[d] & K_0(S) \ar[d] \\
           K_0(A_E) \ar[r]^{\phi_E}&  K_0(S_E),}
\]
where the vertical homomorphisms are induced by the inclusion $F
\subset E$. The bottom horizontal map is an isomorphism, and the left
vertical map is injective. It follows that $\phi$ is injective, and
hence an isomorphism.
\end{proof}

As $\phi$ is an isomorphism, the hypothesis on the variety $V$ in
Proposition \ref{surjection} is always true, we obtain the
following corollary.

\begin{corollary}
$\phi_V: K_0(V;A) \ra K_0 (V \times S)$ is surjective for any
  $F$-variety $V$. \label{surjection2} \hfill $\Box$
\end{corollary}

\section{Higher K-theory}

We have shown that the $K_0$ groups of $S$ and $A$ coincide. We
will show that this is also true for the higher Quillen
$K$-groups.

For any $F$-varieties $X$, $Y$, and $Z$, and separable
$F$-algebras $A$, $B$, and $C$, consider the functor:
\[\mathbf{P}(Y \times Z; B\op \tens_F C) \times \mathbf{P}(X
\times Y; A\op \tens_F B) \ra \mathbf{P}(X \times Z; A\op \tens_F
C),\] sending a pair $(M, N)$ to $(p_{13})_*(p_{23}^*(M) \tens_B
p_{12}^*(N))$, where $p_{12}$, $p_{23}$, and $p_{13}$ are the
projections of $X \times Y \times Z$ onto its factors. This
functor is bi-exact, and thus induces a product map
\[K_n (Y \times Z; B\op \tens_F C) \tens_{\Z} K_m(X \times Y;
A\op \tens_F B) \ra K_{n+m}(X \times Z; A\op \tens_F C).\] We will
denote the image of $u \tens x$ under this map by $u \bullet_B x$.

We recall the $K$-Motivic Category $\cC$ and some of its
properties. The details can be found in \cite{MerPan97} and
\cite{Pan94}. Objects of $\cC$ are pairs $(X, A)$, where $X$ is an
$F$-variety and $A$ is a separable $F$-algebra.  For two pairs
$(X, A)$ and $(Y, B)$ in $\cC$, we set $\Mor_{\cC}((X,A),(Y,B)) :=
K_0(X \times Y; A\op \tens_F B)$. The composition law is $g \circ
f = g \bullet_B f$, for $f: (X, A) \ra (Y, B)$ and $g: (Y, B) \ra
(Z, C)$ in $\cC$.  For any pair $(X,A)$ with $X$ smooth, the
identity element $1_{(X,A)} \in K_0 (X \times X; A\op \tens_F A)$
is the element $[\cO_\Delta \tens_F A]$, where $\Delta \subset X
\times X$ is the diagonal. For any $F$-variety $X$ and any
separable $F$-algebra $A$, we will write $X$ for the pair $(X,F)$
and $A$ for the pair $(\Spec F, A)$. Finally, for any $F$-variety
$V$ and any nonnegative integer $n$, we have a realization functor
$K_n^V$, which sends an object $(X,A)$ to $K_n(V \times X;A)$, and
$K_n^V(f)(x) = f \bullet_A x \in K_n (V \times Y; B)$ for any
morphism $f \in \Mor_{\cC}((X, A), (Y, B)) = K_0 (X \times Y; A\op
\tens_F B)$ and $x \in K_n(V \times X; A)$. We will denote
$K_n^{\Spec F}$ by $K_n$.

As we mentioned in the beginning of Section \ref{secK_0}, there is
a left action of $A\op = \End_{\cO_S}(\cO_S) \times
\End_{\cO_S}(\cI) \times \End_{\cO_S}(\cJ)$ on the locally free
sheaf $\cP = \cO_S \oplus \cI \oplus \cJ$. So $\cP \in \mathbf{P}
(X; A\op)$. The corresponding element $[\cP] \in K_0 (S; A\op)$
defines a morphism $u: A \ra S$ in $\cC$. It follows from the
construction of the realization functor that $\phi_V = K_0^V(u)$
for any $V$. In particular, $K_0 (u) = \phi$.

\begin{theorem}
$u:A \ra S$ is an isomorphism in $\cC$.
\end{theorem}

\begin{proof}
Let $V$ be any $F$-variety. Equating $K_0 (V; A)$ (resp. $K_0 (V
\times S)$) with $\Mor_{\cC} (V, A)$ (resp. $\Mor_{\cC}(V, S)$),
$K_0^V(u):K_0 (V; A) \ra K_0 (V \times S)$ is just
post-composition in $\cC$ with $u$. By Corollary
\ref{surjection2}, $K_0^V(u) = \phi_V$ is surjective for any
variety $V$. In particular, if $V = S$, there is an element $v \in
K_0 (S;A)$ such that $uv = [\cO_\Delta]$, i.e. $u$ has a right
inverse $v$ in $\cC$.

We want to show that $v$ is also a left inverse to $u$ in $\cC$,
i.e. $vu = [A] \in K_0(A\op \tens_F A)$. As $K_0(A\op \tens_F A)
\hookrightarrow K_0((A\op \tens_F A)_{\overline{F}}) = K_0 (
A_{\overline{F}}\op \tens_{\overline{F}} A_{\overline{F}})$, it
suffices to consider the case where $K$, $L$, $B$ and $Q$ are
split. So $A = F \times M_3(F \times F) \times M_2(F \times F
\times F)$, and thus $A$ is isomorphic in $\cC$ to $F \times (F
\times F) \times (F \times F \times F)$ (cf. example 1.6 of
\cite{MerPan97}).  So $K_0(A) \cong \Z^6$, and $\Mor_{\cC}(A,A) =
K_0(A\op \tens_F A) \cong M_6(\Z)$. Moreover, under this
isomorphism $[A] \in K_0(A\op \tens_F A)$ corresponds to the
identity matrix.

So $vu \in K_0 (A\op \tens_F A)$ is represented by a matrix $M$
with integer entries. It follows that the corresponding
homomorphism $K_0(vu)$ from $K_0(A) \cong \Z^6$ to itself is
multiplication by this matrix $M$. Now, as $v$ is a right inverse
to $u$ in $\cC$, $K_0(v)$ is a right inverse to $K_0(u)$. However,
$K_0(u) = \phi$ is an isomorphism by Theorem \ref{zero}, so in
fact $K_0(v) = K_0(u)^{-1}$. Thus $K_0(vu) = K_0(v)K_0(u) =
\id_{K_0(A)}$, which forces $M$ to be the identity matrix. Thus
$vu = [A] \in K_0 (A\op \tens A)$, i.e. $vu = \id_A$ in $\cC$.
\end{proof}

\begin{corollary}
For any integer $n$, any central simple $F$-algebra $D$, and any
$F$-variety $V$,
\[K_n(V; A \tens_F D) \cong K_n(V \times S; D).\] In
particular, $ K_n(F) \oplus K_n(B) \oplus K_n(Q) = K_n(A) \cong K_n(S)$.
\end{corollary}

\begin{proof}
For any central simple $F$-algebra $D$, Morita Equivalence gives a
natural isomorphism $K_0(S; A\op \tens_F D\op \tens_F D) =
K_0(S;A\op)$. Thus the isomorphism $u:A \ra S$ in $\cC$ also
defines an isomorphism from $A \tens_F D$ to $(S,D)$ in $\cC$.
Applying the realization functor $K_n^{V}$ yields $K_n (V; A
\tens_F D) \cong K_n (V \times S; D).$
\end{proof}

We conclude the paper with an Index Reduction Formula for the
function field of the surface $S(B,Q, KL)$. We will need the
following lemma:

\begin{lemma}[\cite{SchVBer92}] Let $X$ be a irreducible $F$-variety,
and $D$ a central simple $F$-algebra. The restriction homomorphism
$K_0(X ;D) \ra K_0(D_{F(X)})$ induced by the inclusion
$\Spec(F(X)) \ra X$ is surjective.
\end{lemma}

\begin{lemma}
Let $X$ be an irreducible $F$-variety, and $D$ a central simple
$F$-algebra.
\[\ind D_{F(X)} = \frac{1}{\deg D}\GCD \{ \rank (P), \forall P \in
\cP (X;D)\}. \] \label{rank}
\end{lemma}

\begin{proof}
We recall that for any field $E$ and any central simple
$E$-algebra $D'$, $K_0(D')$ is cyclic, generated by the class of a
simple $D'$-module $M'$. Moreover, $\dim_E (M') = \deg (D')
\ind(D')$.

The rank homomorphism $\rank: K_0(X;D) \ra K_0(F(X))$ has the
following decomposition:
\[K_0(X;D) \ra K_0(D_{F(X)}) \ra K_0(F(X)),\]
where the first map is induced by the inclusion $\Spec (F(X)) \ra
X$, and the second map takes the class of a $D_{F(X)}$-module to
the class of the corresponding $F(X)$-vector space.

As $K_0(F(X))$ is cyclic, the image of the rank homomorphism is
$n[F(X)]$, where $n$ is the greatest common divisor of the numbers
$\rank (P)$, for all $P \in \cP (X;D)$. By the previous lemma, the
homomorphism $K_0(X;D) \ra K_0(D_{F(X)})$ is surjective. Thus if
$M$ is a simple $D_{F(X)}$-module, \begin{displaymath}n =
\dim_{F(X)}(M) = \deg (D_{F(X)}) \ind(D_{F(X)}) = \deg (D)
\ind(D_{F(X)}),\end{displaymath} and the result follows.
\end{proof}

\begin{corollary}[Index Reduction Formula]
Let $S = S(B, Q, KL)$ be a del Pezzo surface of degree 6. For any
central simple $F$-algebra $D$, $\ind D_{F(S)}$ is equal to:
\begin{enumerate}

\item[i.] $\GCD\{\ind(D), 2\ind(D \tens_F B), 3\ind (D \tens_F Q )
  \}$, if $K$ and $L$ are fields.

\item[ii.]$\GCD\{\ind(D), \ind(D \tens_F B_1), \ind(D \tens_F
B_2)\}$, if $K = F \times F$ and $L$ is a field. Here
  $B = B_1 \times B_2$.

\item[iii.] $\GCD\{\ind(D), \ind (D \tens_F Q_1), 2\ind (D \tens_F
Q_2)\}$, if $K$ is a field, and $L = F \times E$. Here $Q = Q_1
\times Q_2$.

\item[iv.] $\GCD\{\ind(D), \ind (D \tens_F Q_1),
  \ind (D \tens_F Q_2), \ind (D \tens_F Q_3)\}$, if $K$ is a field,
  and $L = F \times F \times F$. Here $Q = Q_1 \times Q_2 \times Q_3$.

\item[v.] $\ind{D}$, when $K$ and $L$ are not fields.
\end{enumerate}
\end{corollary}

\begin{remark}
In case ii., $Q = M_2(L)$ is necessarily split, as $K$ is not a
field. Then $\ind(D \tens_F M_2(L)) = \ind(D_L)$, and as $\ind(D)$
divides $[L:F]\ind(D_L) = 3\ind(D_L)$, the greatest common divisor
will not change if we remove the term $3\ind (D \tens_F Q)$.
Similarly, in cases iii., iv., and v., we can remove the term with
the split $B$ or $Q$ when computing greatest common divisors.
\end{remark}

\begin{proof}
As $u:A \ra S$ is an isomorphism in $\cC$, it defines an isomorphism
$K_0(u)$ from $K_0(A \tens_F D)$ to $K_0(S; D)$. Moreover, as $A = F
\times B \times Q$, $K_0(A \tens_F D) \cong K_0(D) \oplus K_0(B
\tens_F D) \oplus K_0(Q \tens_F D)$.

We will consider the case where $K$ and $L$ are fields. The proof
of the other cases are similar. As $D$, $B \tens_F D$, and $Q
\tens_F D$ are central simple algebras (with centers $F$, $K$, and
$L$, respectively), their $K_0$ groups are cyclic, generated by
the class of a simple module. Therefore by Lemma \ref{rank},
$\deg(D) \ind (D_{F(S)})$ will equal the greatest common divisor
of the ranks of the images of simple $D$, $B \tens_F D$ and $Q
\tens_F D$ modules under the image of $K_0(u): K_0 (A \tens_F D)
\ra K_0 (S;D)$.

So let $M_B$ be a simple $B \tens_F D$-module. Then $\dim_K (M_B)
= \deg (B \tens_F D)\ind(B \tens_F D)$, and thus
\begin{align*}
\rank (K_0(u)(M_B)) & = \rank (M_B \tens_B \cI) \\
& = \frac{\dim_F (M_B) \rank(\cI)}{\dim_F(B)} \\
& = \frac{\dim_K (M_B) \rank(\cI)}{\dim_K(B)} \\
& = \frac{\deg(B \tens_F D)\ind(B \tens_F D)\rank(\cI)}{\dim_K(B)} \\
& = 2 \deg(D)\ind (B \tens_F D).
\end{align*}
Similarly, if $M_Q$ (resp. $M_F$) is a simple $Q \tens_F D$-module
(resp. $D$-module), $\rank (K_0(u)(M_Q)) = 3\deg(D)\ind (D \tens_F
Q)$ (resp. $\rank(K_0(u)(M_F)) = \deg(D)\ind(D)$), and the result
follows.
\end{proof}

\end{document}